\documentclass{article}
\usepackage{euler}

\usepackage{amsmath,amssymb,diagrams}

\usepackage[pdftex,
            pdfauthor={Joshua Mundinger},
            pdftitle={Quantization of restricted Lagrangian subvarieties in positive characteristic},
			]{hyperref}

\usepackage[
	backend=bibtex,
	style=alphabetic,
	citestyle=alphabetic,
	maxnames=100 
]{biblatex}
\usepackage{amsthm}

\theoremstyle{plain}
\newtheorem{theorem}{Theorem}[section]
\newtheorem{lemma}[theorem]{Lemma}
\newtheorem{proposition}[theorem]{Proposition}

\newtheorem{example}[theorem]{Example}

\theoremstyle{definition}

\newtheorem{definition}[theorem]{Definition}
\theoremstyle{remark}
\newtheorem{remark}[theorem]{Remark}

\newcommand{\Aa}{\mathcal{A}}
\newcommand{\GG}{\mathbf{G}}
\newcommand{\OO}{\mathcal{O}}
\newcommand{\LL}{\mathcal L}
\newcommand{\mm}{\mathfrak m}
\newcommand{\nn}{\mathfrak n}

\newcommand{\pop}{ {[p]} }		

\newcommand{\Het}{H_{\acute{e}t}} 	

\newcommand{\qcoords}{\mathcal P}	

\newcommand{\psup}[1]{\widetilde{ #1'}}
\newcommand{\aclasscomplex}[1]{At_{#1}}

\newcommand{\obstr}[1]{\mathfrak{o}_{#1, \OO_h}}
\newcommand{\oflat}{\mathfrak{o}^\flat_{Y}}

\newcommand{\series}[1]{\left[  \left[ #1 \right]\right]}
\newcommand{\lseries}[1]{\left( \!\left( #1\right)\!\right)}

\newcommand{\tensor}{\otimes}

\DeclareMathOperator{\Res}{Res}
\DeclareMathOperator{\Spec}{Spec}

\title{Quantization of restricted Lagrangian subvarieties in positive characteristic}
\author{Joshua Mundinger}
\date{October 15, 2022}

\begin{document}
\maketitle

\abstract{
	Bezrukavnikov and Kaledin introduced quantizations of symplectic varieties $X$ in positive characteristic which endow the Poisson bracket on $X$ with the structure of a restricted Lie algebra.
	We consider deformation quantization of line bundles on Lagrangian subvarieties $Y$ of $X$ to modules over such quantizations.
	If the ideal sheaf of $Y$ is a restricted Lie subalgebra of the structure sheaf of $X$,
	we show that there is a certain cohomology class which vanishes if and only if a line bundle on $Y$ admits a quantization.
}

\section{Introduction}

Fix a field $k$ of characteristic $p>2$, 
and let $(X,\omega)$ be a smooth variety over $k$ equipped with a symplectic form.
Unlike in characteristic zero, the Poisson bracket $\{-,-\}$ on $\OO_X$ has a large center: it follows from the Leibniz rule that $\{f^p,g\} = 0$ for all sections $f,g$ of $\OO_X$.
Bezrukavnikov and Kaledin studied certain quantizations $\OO_h$ of the Poisson sheaf $\OO_X$, known as \emph{Frobenius-constant quantizations}, where the quantization also has a large center \cite{bk}.
More precisely, the relative Frobenius map $\OO_{X'} \to \OO_X$ lifts to an inclusion 
$s: \OO_{X'} \to \OO_h$ into the center of $\OO_h$,
inducing an isomorphism $Z(\OO_h) \cong \OO_{X'}\series{h}$,
where $'$ indicates Frobenius twist, here and throughout the paper.
In fact, this map is a lift of the $p$-th power map $\OO_X' \to \OO_h/h^{p-1}$.
The lift $s: \OO_{X'} \to \OO_h$ makes the Poisson bracket on $\OO_X$ into a restricted Lie algebra
via the $p$-operation
\begin{equation} \label{p-operation-definition}
	 f^\pop = \frac{f^p - s(f\tensor 1)}{h^{p-1}}.
\end{equation}
Frobenius-constant quantizations have been used to construct derived equivalences associated to symplectic resolutions \cite{bkresolutions, kaledin08}.

We consider deformation quantization of modules over $\OO_X$ to modules over such $\OO_h$. 
In characteristic zero, Gabber's celebrated integrability of characteristics theorem implies that a coherent sheaf admitting a quantization is supported on a coisotropic subvariety \cite{gabber}.
We specifically consider quantizing modules of the form $i_*\LL$ where $i: Y \to X$ is the inclusion of a smooth Lagrangian subvariety and $\LL$ is a line bundle on $Y$.
Since $\OO_h$ has a large center, a quantization $\LL_h$ has associated to it its \emph{$p$-support} $\psup{Y}$, its support in $Z(\OO_h) \cong \OO_{X'}\series{h}$.
When we investigate the existence of a quantization of a line bundle, the $p$-support will be given data, and we will ask for conditions for a quantization with the given $p$-support to exist.

The basic example of such a quantization is over differential operators.
If $X = T^\ast Y$, $i: Y \to X$ is the zero section, and $\OO_h = D_{Y,h}$ is the $h$-crystalline differential operators on $Y$,
then for every formal series of closed $1$-forms $\alpha \in \Omega^1_Y\series{h}$,
there is a quantization of $\OO_Y$ given by the integrable $h$-connection $\nabla = hd  +h\alpha$ on $\LL_h = \OO_Y\series{h}$.
The action of the center $\OO_{T^\ast Y}\series{h}$ sends $\partial \in T_Y$
to the $p$-curvature 
\begin{align*}
	\nabla_\partial^p - h^{p-1}\nabla_{\partial^\pop} &=
	(h\partial^\pop + h \alpha(\partial))^p - h^p(\partial^\pop + \alpha(\partial^\pop)) \\
	&= h^p (\alpha(\partial)^p + \partial^{p-1}\alpha(\partial) - \alpha(\partial^\pop)), 
\end{align*}
which is divisible by $h^p$. The $p$-support $\psup{Y}$ is the graph of this $p$-curvature in $T^\ast Y'\series{h}$, which in this case is a deformation of the zero section $Y'\subseteq T^\ast Y'$ which is trivial modulo $h^p$.

The zero section of the cotangent bundle enjoys a certain compatibility with the restricted structure.

\begin{definition} \label{definition: restricted-subvariety}
	A coisotropic subvariety $Y \subseteq X$ is called \emph{restricted} if its ideal sheaf is closed under the 
	$p$-operation $f \mapsto f^\pop$.
\end{definition}

This paper in large part explores the geometry of smooth restricted Lagrangian subvarieties. 
We will show that smooth restricted Lagrangian subvarieties are, in an appropriate sense, locally isomorphic to the zero section of the cotangent bundle (see Theorem \ref{theorem: tubular-neighborhood} for the precise statement).

The methods of the Gelfand-Kazhdan formal geometry apply to analyze the existence of quantizations. 
We construct a certain class $\obstr{\psup{Y}} \in H_{fl}^2(\psup{Y}, \mathbb G_m)$ below, depending on $Y$, $\OO_h$, and the $p$-support $\psup{Y}$.

\begin{theorem}\label{theorem: existence-of-quantization}
	Let $Y \subseteq X$ be a smooth restricted Lagrangian subvariety, $\psup{Y} \subseteq X'\series{h}$ be a deformation of $Y' \subseteq X'$
	which is trivial modulo $h^p$, and $\OO_h$ a Frobenius-constant quantization of $X$.
	Then there exists a line bundle $\LL$ on $Y$ and a quantization $\LL_h$ of $\LL$ over $\OO_h$ with p-support $\psup{Y}$
	if and only if 
	\[ \obstr{\psup{Y}} = 1.\]
\end{theorem}

In case the $p$-support deforms trivially, so that $\psup{Y} = Y'\series{h}$, we compare this class to a certain Brauer class $[\OO_h^\sharp] \in H_{fl}^2(X'\series{h}, \mathbb G_m)$ introduced by Bogdanova and Vologodsky \cite{bv20}.
We view the assignment $\OO_h \mapsto [\OO_h^\sharp]$ as a positive-characteristic analog of the noncommutative period map.
We will review the construction of $[\OO_h^\sharp]$ below in §\ref{ss: bv}. 

\begin{theorem} \label{theorem: comparison-to-bv}
	Let $Y \subseteq X$ be a smooth restricted Lagrangian subvariety and $\psup{Y} = Y'\series{h}$.
	For all Frobenius-constant quantizations $\OO_h$ of $X$,
	\[ \obstr{Y'\series{h}} = [\OO_h^\sharp]_{Y'} \in H_{fl}^2(Y', L^+\mathbb G_m).\]
\end{theorem}

Theorem \ref{theorem: existence-of-quantization} is analogous to the main theorem of \cite{bgkp}, where in characteristic zero a quantization of $Y$ exists if and only if the Deligne-Fedosov class associated to the quantization, also known as the noncommutative period, vanishes on $Y$.
Indeed, our approach is similar to and inspired by \cite{bgkp}. One of the new features of our approach, following \cite{bk}, is that torsors over certain nonreduced group schemes replace the Harish-Chandra torsors used in characteristic zero.

Theorem \ref{theorem: existence-of-quantization} describes when some line bundle on $Y$ may be quantized.
A necessary condition for a particular line bundle $\LL$ to be quantized 
involves a certain positive-characteristic refinement $c_r$ of the Chern class, which takes values in $\Het^1(\Omega^1_{log})$.
It also depends on a certain class $\rho(\OO_h)$ which classifies the first-order quantization $\OO_h/h^2\OO_h$ (see Proposition \ref{prop: first-order-torsor}).
We will show in Theorem \ref{theorem: chern-class-of-quantization}
that if $\LL$ admits a quantization, then 
\[ 
	c_r(\LL) = \rho(\OO_h)|_Y + \frac{1}{2} c_r(K_Y) + [i_\theta \omega'],
\]
where $[i_\theta \omega']$ is a certain class describing the first nontrivial order of deformation of the $p$-support. 
We will also show that this condition is sufficient for $\LL$ to admit a quantization if $Pic(\psup{Y}) \to Pic(Y')$ is onto (for instance, if $\psup{Y} = Y'\series{h})$).

\paragraph{Acknowledgments.}{
The author thanks Victor Ginzburg for invaluable advice and conversations. Ekaterina Bogdanova, Roman Travkin, and Vadim Vologodsky made insightful comments on earlier versions of this work.  
The author thanks the anonymous referee for their comments and for their improvement of the statement of Theorem \ref{theorem: comparison-to-bv}. The author was supported by the NSF Graduate Research Fellowship DGE 1746045.
}

\section{Geometry of restricted Lagrangian subvarieties}

From this point forward, all Lagrangian subvarieties considered will be smooth.

\subsection{Preliminaries on restricted structures}

We begin by recalling Bezrukavnikov and Kaledin's definition of a restricted structure on a symplectic variety $(X,\omega)$. 
We will also need the notion of a restricted structure on a quantization. 
The notion of quantized algebra, defined below, provides a common framework for discussing restricted structures in these contexts.
\begin{definition}\cite[Definition 1.5]{bk}
	A \emph{quantized algebra} $A$ is an associative $k\series{h}$-algebra equipped with a $k\series{h}$-linear Lie bracket $\{-,-\}$ which is a derivation in each variable and satisfies $h\{x,y\} = xy-yx$ for all $x,y \in A$.
\end{definition}
A quantized algebra with $h=0$ is a Poisson algebra over $k$, while a quantized algebra which is flat over $k\series{h}$ is an associative $k\series{h}$-algebra.

There is a certain universal quantized polynomial $P$ which measures the failure of the Frobenius to be multiplicative in a quantized algebra. It satisfies 
\[ h^{p-1}P(x,y) = (xy)^p - x^py^p \]
for any $x$ and $y$ in a quantized algebra \cite[(1.3)]{bk}.
\begin{definition}
	A \emph{restricted structure} on a quantized algebra $A$ is an operation $x \mapsto x^\pop$ on $A$ satisfying:
	\begin{itemize}
		\item $\langle A, \{-,-\}, -^\pop \rangle$ is a restricted Lie algebra over $k$;
		\item $h^\pop = h$;
		\item $(xy)^\pop = x^py^\pop + x^\pop y^p - h^{p-1} x^\pop y^\pop + P(x,y)$ for all $x,y \in A$.
	\end{itemize}
\end{definition}

The first step towards the construction of a Frobenius-constant quantization 
is the construction of a restricted structure on the Poisson sheaf $\OO_X$,
which is called a restricted structure on $X$. 
\begin{definition}
	Let $Z/k$ be a smooth variety. For a vector field $\partial$,
	the \emph{restricted contraction} by $\partial$ is the operation 
	$i^\pop_\partial:\Omega^{\ast + 1}_Z \to \Omega^{\ast}_Z $ defined by 
	\[
		 i^\pop_\partial: \alpha \mapsto i_{\partial^\pop}\alpha - L_{\partial}^{p-1} i_\partial \alpha,
	\]
	where $L_\partial$ is the Lie derivative with respect to $\partial$.
\end{definition}

Let $\Omega^{\leq 1}$ be the de Rham complex truncated below degree $1$.
\begin{theorem}\cite[Theorem 1.12]{bk} \label{theorem: restricted-structure-from-form}
	A restricted structure on a symplectic variety $(X,\omega)$ 
	is equivalent to a choice of $[\eta] \in H^1(\Omega_X^{\leq 1})$ such that $d[\eta] = \omega$.
	Given $[\eta]$, the restricted operation sends $f \in \OO_X$ with Hamiltonian vector field $H_f$ to
	\[
		f^\pop = i_{H_f}^\pop \eta, 
	\]
	where $\eta$ is a 1-form locally representing $[\eta]$.
\end{theorem}
\begin{remark}
	Even if the symplectic form is locally exact, such a class $[\eta]$ need not exist, e\.g\. in the case of an abelian variety. 
\end{remark}
The key lemma about restricted contraction is on its relationship with the Cartier operator $C$.
\begin{lemma} \cite[Lemma 2.1]{bk} \label{lemma: cartier-and-restricted-derivative}
	Let $Z/k$ be a smooth variety and $\alpha$ a closed differential form on $Z$.
	Then for all vector fields $\partial$,
	\begin{equation*}
		C(i^\pop_\partial \alpha) = i_{\partial'} C(\alpha),
	\end{equation*}
	where $\partial'$ is the corresponding vector field on $Z'$.
\end{lemma}

\subsection{Restricted Lagrangian subvarieties}

Recall from Definition \ref{definition: restricted-subvariety} that a coisotropic subvariety is \emph{restricted} if its ideal sheaf is closed under the restricted operation.

\begin{proposition}
	For a Lagrangian subvariety $Y$ of a restricted symplectic variety $(X,[\eta])$, the following are equivalent:
	\begin{enumerate}
		\item the ideal sheaf of $Y$ is stable under the restricted operation;
		\item $[\eta]_Y = 0$ in $H^1(Y, \Omega^{\leq 1})$.
	\end{enumerate}
\end{proposition}
\begin{proof}
	The question is local, so we may assume that $[\eta]$ is represented by a global 1-form $\eta$.
	Let $I$ be the ideal sheaf of $Y$,
	and let $\eta_Y$ denote the restriction of $\eta$ to $Y$.
	Since $Y$ is Lagrangian and $d\eta = \omega$, $\eta_Y$ is closed.
	We wish to show that $\eta_Y$ is locally exact if and only if $I^\pop \subseteq I$.

	Theorem \ref{theorem: restricted-structure-from-form} states that for all local sections $f$ of $\OO_X$, the restricted operation is given by
	\[ 
		f^\pop = i^\pop_{H_f}\eta.
	\]
	Hence $f^\pop \in I$ for all $f \in I$ if and only if $i^\pop_{H_f}(\eta_Y) = 0$ for all $f \in I$.
	As $Y$ is Lagrangian, the set of Hamiltonian vector fields $\{H_f \mid f \in I\}$ spans $T_Y$, so $I^\pop \subseteq I$ if and only if 
	\begin{equation*}
		 i^\pop_{\partial}(\eta_Y) = 0
	\end{equation*}
	for all local vector fields $\partial$ on $Y$. By Lemma \ref{lemma: cartier-and-restricted-derivative} and the Cartier isomorphism,
	this is equivalent to that $\eta_Y$ is locally exact.
\end{proof}

\begin{example}
	The canonical $1$-form $\lambda$ on the cotangent bundle $T^*X$ gives $T^*X$ a restricted structure.
	A section of the cotangent bundle $s_\alpha:X\to T^*X$ corresponding to a $1$-form $\alpha$
	satisfies $s_\alpha^*\lambda = \alpha$.
	Hence the graph of $s_\alpha$ is Lagrangian if and only if $\alpha$ 
	is closed, while it is restricted Lagrangian if and only if $\alpha$
	is locally exact.
\end{example}

\subsection{Local normal form}

Our goal is to find a local normal form for restricted Lagrangian subvarieties.
In the smooth category, Weinstein's tubular neighborhood theorem states that for a Lagrangian submanifold $L \subseteq M$, every point in $L$ has a neighborhood in $M$ which is symplectomorphic to a neighborhood of the zero section of the cotangent bundle of $L$ \cite[Lecture 5]{weinstein77}.
We establish in this section a kind of tubular neighborhood theorem for restricted Lagrangian subvarieties in positive characteristic. In this setting, our neighborhood will be a neighborhood in the fpqc topology, as in Bezrukavnikov and Kaledin's version of the Darboux theorem 
\cite[Theorem 3.4]{bk}. 

\begin{definition}
	A \emph{Frobenius-constant quantization} of a restricted symplectic variety $X$ is a sheaf $\OO_h$ of flat $k\series{h}$-algebras, complete with respect to the $h$-adic filtration, and a map of algebras $s: \OO_X \to \OO_h$
	which is $k$-Frobenius-linear and satisfies $s(f) \cong f^p \mod h^{p-1}$,
	such that with the $p$-operation
	\[ 
		f^\pop = \frac{f^p - s(f)}{h^{p-1}}
	\]
	on $\OO_h$,
	there is an isomorphism of restricted Poisson algebras $\OO_h/h\OO_h \cong \OO_X$.
\end{definition}
The center of a Frobenius-constant quantization is isomorphic via $s$ to $\OO_{X'}\series{h}$ \cite[Lemma 1.10]{bk}. 

The local model for a Frobenius-constant quantization of a symplectic variety of dimension $2n$ is the reduced Weyl algebra $A_h$ in $2n$ variables, defined as follows:
it is the $k\series{h}$-algebra with generators $x_1,\ldots, x_n$ and $y_1,\ldots, y_n$ and relations
\begin{equation*}
	[x_i,x_j] = [y_i,y_j] = x_i^p = y_i^p = 0, \qquad
	[y_i,x_j] = \delta_{ij}h
\end{equation*}
for all $i$ and $j$.
The restricted Weyl algebra is the unique Frobenius-constant quantization of the Frobenius neighborhood
\[ A_0= k[x_1,\ldots, x_n,y_1,\ldots, y_n]/(x_1^p,\ldots, y_n^p)\]
with restricted structure given by the 1-form 
\[\eta = \sum_i y_i dx_i.\]
Bezrukavnikov and Kaledin showed that a restricted symplectic variety is 
locally isomorphic to $(\Spec A_0,\eta)\times X'$ in the fpqc topology on $X'$,
and that every Frobenius-constant quantization is in the same sense
locally isomorphic to $A_h$ \cite[Theorem 3.4]{bk}.
That is a positive-characteristic version of Darboux's theorem.

The Frobenius neighborhood $(\Spec A_0,\eta)$ may be thought of as a subscheme of the cotangent bundle to $\Spec k[x_1,\ldots, x_n]/(x_1^p,\ldots, x_n^p)$. 
Our local model for a restricted Lagrangian subvariety is then the zero section of this cotangent bundle, defined by the following ideal:
\begin{equation}
	J = (h,y_1,\ldots, y_n) \subseteq A_h.
\end{equation}
Now we may prove our local tubular neighborhood theorem. 
\begin{theorem} \label{theorem: tubular-neighborhood}
	Let $Y \subseteq X$ be a restricted Lagrangian subvariety,
	$\OO_h$ be a Frobenius-constant quantization of $X$,
	and $I$ be the ideal of $Y$ in $\OO_h$.
	If $\psup{Y} \subseteq X'\series{h}$ is a formal deformation of $Y'\subseteq X'$ which is trivial modulo $h^p$,
	then fpqc locally on $\psup{Y}$, there are isomorphisms of restricted quantized algebras $\OO_h|_{\psup{Y}} \cong A_h \tensor_{k\series{h}} \OO_{\psup{Y}}$ taking $I|_{\psup{Y}}$ to $J\tensor \OO_{\psup{Y}}$. 
\end{theorem}
\begin{proof}	
	Let $\mathfrak m$ denote the maximal ideal of $A_h$.
	The strategy is first to put both $\OO_h$ and $Y$ into Frobenius-local coordinates,
	then twist so that $\mm$ maps into the maximal ideal of the smaller Frobenius neighborhood.
	This is where the hypothesis on $\psup{Y}$ is used.
	Finally, a semisimple-by-unipotent method takes the kernel of the projection $\OO_h \to \OO_Y$ to $J$.

	Let $B = k[z_1,\ldots, z_n]/(z_1^p,\ldots, z_n^p)$ 
	with maximal ideal $\nn$.
	We may view $B$ as a $k\series{h}$-algebra where $h=0$.
	Then $\OO_Y$ is locally isomorphic to $B \tensor_k \OO_{Y'}$ over $Y'$ in the fpqc topology.
	Since $Y'/k$ is smooth, the deformation $\psup{Y}$ is locally trivial,
	so we may take a Zariski-open cover of $\psup{Y}$ 
	which is the pullback along $\Spec k\series{h} \to \Spec k$ of a Zariski cover of $Y'$.
	Refining this cover, we may find an fpqc cover $U \to \psup{Y}$
	such that $\OO_h|_U \cong A_h \tensor_{k\series{h}} \OO_U$ and
	$\OO_Y|_U \cong B \tensor_{k\series{h}} \OO_U$.
	Further, since 
	\[ \psup{Y} \times_{\Spec k\series{h}} \Spec k[h]/h^p \cong Y[h]/h^p,\]
	we may assume that this holds for $U$ also.
	
	Let $\Spec R$ be an affine open in $U$. The map $\OO_h \to \OO_Y$ induces a surjective map
	\[ \psi: A_h \tensor_{k\series{h}} R \to B \tensor_{k\series{h}} R,\]
	and we wish to show that this may be twisted to a map with kernel $J \tensor_{k\series{h}} R$.
	Let $\epsilon: B \to k$ be the augmentation, and let 
	\begin{align*}
		a_i &= (\epsilon \tensor 1)\psi(x_i), \\
		b_j &= (\epsilon \tensor 1)\psi(y_j)
	\end{align*} for $1 \leq i,j \leq n$.
	These elements of $R/hR$ satisfy $a_i^p = b_j^p = 0$ for all $i$ and $j$.
	Since $R/h^pR \to R/hR$ has a section, there are lifts $\tilde a_i, \tilde b_j$ such that $\tilde a_i^p, \tilde b_j^p \in (h^p)$ for all $i$ and $j$.
	Let $R'$ be an fppf $R$-algebra such that there are elements $c_i,d_j \in R'$ satisfying
	\begin{align*}
		h^pc_i^p &= \tilde a_i^p, 	\\
		h^p d_j^p &= \tilde b_j^p.
	\end{align*}
	There is an automorphism $\phi \in Aut(A_h)(R')$ given by
	\begin{align*}
		\phi(x_i) &= x_i - a_i + hc_i, \\
		\phi(y_j) &= y_j - b_j + hd_j .
	\end{align*}
	The morphism $\psi'= \psi\phi$ satisfies $(\epsilon \tensor 1)(\psi'(x_i)) = (\epsilon\tensor 1)(\psi'(y_j)) = 0$ for all $i$ and $j$,
	and hence takes $\mm_{R'}$ into $\nn_{R'}$.

	We may now assume $R' =R$ and $\psi(\mm_R) \subseteq \nn_R$.
	Since $\psi$ is surjective, the induced $R$-module map 
	\[ \mm_R/(h+\mm_R^2) \to \nn_R/\nn_R^2\]
	must be also. The kernel is a Lagrangian subspace of the symplectic vector bundle $\mm_R/(h+\mm_R^2)$ over $\Spec R$. 
	Hence Zariski locally on $R$ there is a symplectic transformation taking this map to the standard map, 
	so that we may assume
	\begin{align*}
		x_i &\mapsto z_i + \nn_R^2,	\\
		y_i &\mapsto 0+ \nn_R^2.
	\end{align*}
	Since the $z_i$ generate $B \tensor_{k\series{h}}R$ as an algebra, it follows that the images of $x_i$ do also.
	Hence we may write $\psi(y_i) = g_i(\psi(x_1),\ldots, \psi(x_n))$ for some 
	$g_i \in (t_1,\ldots, t_n)^2R[t_1,\ldots,t_n]$.
	It follows from the standard presentation of $A_h$ that 
	\[\ker \psi = (h, y_1 - g_1(x_1,\ldots, x_n), y_n - g_n(x_1,\ldots, x_n)).\]

	Let $\overline{g_i}$ denote the reduction of $g_i$ modulo $h$,
	and consider 
	\[\alpha = \sum_{i=1}^n \overline{g_i}(z_1,\ldots, z_n) dz_i \in \Omega^1_{B/k}\tensor_k R/hR.\]
	$B$ is isomorphic to the $k$-subalgebra of $A_h$ generated by the $x_i$'s,
	so we may also embed it into $A_h$.
	Since $Y$ is coisotropic, $\alpha$ is closed.
	Since $Y$ is restricted, $\ker \psi$ is closed under the restricted power.
	We show $C(\alpha) = 0$
	where $C$ is the Cartier operator.
	Let $L$ be Jacobson's Lie polynomial which measures the failure of the restricted operation to be additive. Then 
	\begin{align*}
		(y_i - g_i(x_1,\ldots, x_n))^\pop &= L(y_i,-g_i(x_1,\ldots, x_n)) \\
					&= -\partial_{x_i}^{p-1} g_i(x_1,\ldots,x_n) 
	\end{align*}
	which reduces modulo $h$ to $i_{\partial_{x_i}}C(\alpha)$.
	But $\ker \psi$ intersected with the subalgebra generated by $\{x_1,\ldots,x_n\}$ is $(h)$,
	showing $C(\alpha) = 0$.
	
	By the Cartier isomorphism for $A_0$, we conclude that there exists
	$\overline{f} \in B\tensor_k R/hR$ such that $\alpha = d\overline{f}$,
	which we may lift to $f\in B\tensor_k R \subseteq A_h \tensor_{k\series{h}} R$.
	As a polynomial in the $x_i$'s, $f$ Poisson commutes with $x_i$ for all $i$,
	while by definition 
	\[
		\{f, y_i\} - g_i(x_1,\ldots, x_n) \in (h).
	\]
	Further, since $\overline{g_i} \in (x_1,\ldots,x_n)^2$,
	we may choose $f \in (x_1,\ldots, x_n)^2$ also. 
	Hence conjugation by $e(f/h)$ for $e$ the restricted exponential
	defines an automorphism of $A_h \tensor_{k\series{h}} R_h$ sending
	$\ker \psi$ to $(y_1,\ldots,y_n,h)$.
\end{proof}

\section{The obstruction to quantization}

Let $Y \subseteq X$ be a restricted Lagrangian subvariety.
\begin{definition}
	A \emph{quantization} of a line bundle $\LL$ on a Lagrangian subvariety $Y \subseteq X$
	is an $\OO_h$-module $\LL_h$ which is flat and complete over $k\series{h}$
	such that $\LL_h/h\LL_h$ is isomorphic to the direct image of $\LL$.
\end{definition}

\begin{definition}
	The \emph{$p$-support} of an $\OO_h$-module $\LL_h$ is the support of $\LL_h$ in $\Spec Z(\OO_h) = X'\series{h}$.
\end{definition}

\begin{proposition}
	Let $\LL_h$ be a quantization of a line bundle on a restricted Lagrangian subvariety $Y$.
	Then the $p$-support of $\LL_h$
	is a formal deformation of $Y' \subseteq X'$.
	Further, this deformation is trivial modulo $h^p$.
\end{proposition}
\begin{proof}
	Since $\LL_h$ is flat over $k\series{h}$, 
	its support is flat over $k\series{h}$ also.
	Let $I$ be the ideal of $Y$ in $\OO_h$.
	For a section $f'$ of $\OO_{X'}$, 
	$s(f') \LL_h \subseteq h \LL_h$ if and only if $s(f') \equiv (f')^p \mod h$ is in $I$,
	which occurs if and only if $f'$ is in the ideal of $Y'$.
	Hence, the support is a formal deformation of $Y'$.
	
	Let $f$ be a section of the ideal of $Y$ in $\OO_{X}$,
	and set $f' = f \tensor 1$ in $\OO_{X'}$.
	Lift $f$ to a section $\tilde f$ of $\OO_h$.
	Then $s(f') \equiv \tilde f^p  - h^{p-1} (\tilde f)^\pop \mod h^p$.
	Now $\tilde f \LL_h \subseteq h \LL_h$,
	and since $Y$ is restricted, $\tilde f^\pop \LL_h \subseteq h \LL_h$.
	We conclude that $s(f') \LL_h \subseteq h^p \LL_h$.
	Thus, the support modulo $h^p$ is exactly $Y'[h]/h^p$,
	so the deformation is trivial modulo $h^p$, as desired.
\end{proof}	

\begin{example}
	The $p$-support of a quantization of a Lagrangian subvariety need not be Lagrangian. Consider $Y = \mathbb{A}^2$ with coordinates $x_1$ and $x_2$,
	$X = T^*Y$ with dual coordinates $y_1$ and $y_2$ to $x_1$ and $x_2$, and $\OO_h = D_{Y,h}$ the crystalline differential operators on $Y$.
	Then the module $\OO_{Y}\series{h}$ with $h$-connection $hd + h x_1^px_2^{p-1}dx_2$ is a quantization of $\OO_Y$.
	The $p$-support is 
	\[ \psup{Y} = V( y_1^p, y_2^p - h^p(x_1^{p^2}x_2^{p(p-1)} - x_1^p)) \subseteq \Spec k[x_1^p,x_2^p,y_1^p,y_2^p].\]
	This subvariety is not coisotropic since $\{y_1^p, x_1^{p^2}x_2^{p(p-1)} - x_1^p\}$ is a unit.
\end{example}

From now on in this section, we fix a formal deformation $\psup{Y}$ of $Y' \subseteq X'$ which is trivial modulo $h^p$, and analyze the existence of a quantization of $Y$ with $p$-support $\psup{Y}$ locally on $\psup{Y}$.

\subsection{Local analysis of quantizations}

\begin{lemma}\label{lemma: unique-local-quantization}
	Let $R$ be a flat $k\series{h}$-algebra and $A_R = A_h \tensor_{k\series{h}} R$.
	There exists a left $A_R$-module $M_R$, unique up 
	to isomorphism, such that $M_R$ is flat over $R$
	and $M_R/hM_R \cong A_R/J$. Further, the automorphisms of $M_R$ are exactly the units $R^\times$ of $R$.
\end{lemma}
\begin{proof}
	Existence is certified by the module $M_R = A_R/A_R(y_1,\ldots, y_n)$.
	Now suppose that $N_R$ is another such module.
	By Nakayama's Lemma, $N_R$ is a free module of rank one over $R[x_1,\ldots,x_n]/(x_1^p,\ldots, x_n^p) \subseteq A_h$.
	Call its generator $1_N$; then 
	\[y_i 1_N = h \alpha_i 1_N\]
	for unique $\alpha_i \in R[x_1,\ldots, x_n]/(x_1^p,\ldots, x_n^p)$, as $R$ is flat over $k\series{h}$.
	Set $\alpha = \sum_i \alpha_i dx_i$.
	The relations $[y_i,y_j] = 0$ imply that $\alpha$ is a closed 1-form.
	For all $f \in R[x_1,\ldots, x_n]/(x_1^p,\ldots, x_n^p)$,
	\[ y_i (f 1_N) = ( h\partial_{x_i}+ h\alpha_i)(f) 1_N,\]
	so the relation $y_i^p = 0$ implies that
	\[ 0 = (h\partial_{x_i} + h\alpha_i)^p =h^p\left( \alpha_i^p + \partial_{x_i}^{p-1}\alpha_i \right)= h^pi_{\partial_{x_i}'}( \alpha' - C(\alpha)),\]
	where $C$ is the Cartier operator.
	Since this holds for all $i$, we conclude $\alpha' = C(\alpha)$,
	so $\alpha$ is logarithmic.
	If $\alpha = dg/g$, then the generator $g^{-1} 1_N$ of $N_R$ is annihilated by $y_i$ for all $i$,
	and hence sending $1_M \to g^{-1}1_N$ defines an isomorphism $M_h \cong N_h$.

	Now suppose that $\varphi: M_R \to M_R$ is an automorphism. Set
	\[ \varphi(1) = \sum_{\beta\in \{0,\ldots,p-1\}^n} c_\beta x_1^{\beta_1}\cdots x_n^{\beta_n}\]
	for some $c_\beta \in R$. 
	The relations $y_i\cdot 1 = 0$ imply $h c_\beta \beta_i = 0$ for all $i$. 
	Since $R$ is flat over $k\series{h}$, we conclude that $c_\beta = 0$ for $\beta \neq 0$,
	from which it follows that $c_0\in R$ is a unit and $\varphi(m) = c_0 m$ for all $m \in M_R$.
\end{proof}

\subsection{Torsors and quantization}\label{ss: torsors-and-quantization}

Theorem \ref{theorem: tubular-neighborhood} shows that restricted Lagrangian subvarieties are locally homogeneous,
and quantizations are also locally homogeneous by Lemma \ref{lemma: unique-local-quantization}. 
The question of whether $Y$ may be quantized will be converted into the question of whether a certain torsor lifts over a central extension.

\begin{definition} \label{automorphism-group}
	\. \\
	\begin{itemize}
		\item	Let $\GG$ be the group of automorphisms of the restricted quantized algebra $A_h$. It is an affine group scheme over $k\series{h}$.
		\item The \emph{torsor of quantized coordinates} $\qcoords$ is the $\GG$-torsor on $X'\series{h}$ of local isomorphisms of $\OO_h$ with $\OO_{X'\series{h}} \tensor_{k\series{h}} A_h$.
	\end{itemize}
\end{definition}
This torsor of quantized coordinates appears in \cite[Lemma 4.3]{bk}.

\begin{remark}
	In \cite{bk}, Bezrukavnikov and Kaledin work with the restriction of scalars of $\GG$ along $k \to k\series{h}$, which is sufficient for analyzing $\GG$-torsors over $X'\series{h}$ associated to a Frobenius-constant quantization.
	However, the $p$-support $\psup{Y}$ may be a nontrivial deformation of $Y'$;
	thus, we must work with schemes and torsors over $k\series{h}$.
\end{remark}

\begin{definition}
	Let $\GG_J\subseteq \GG$ be the fpqc sheaf of stabilizers of the ideal $J = (h,y_1,\ldots, y_n)$.
\end{definition}

\begin{remark}
	$\GG_J$ is not representable by a scheme of finite type over $k\series{h}$. 
	Since $J(h^{-1}) = A_h(h^{-1})$, the fiber of $\GG_J$ over $k\lseries{h}$ agrees with that of $\GG$,
	while the fiber over $k\series{h}/(h)$ is smaller in dimension.
	Nonetheless we may consider $\GG_J$-torsors.
\end{remark}

\begin{proposition} \label{proposition: restriction-to-GJ}
	Let $Y \subseteq X$ be a Lagrangian subvariety.
	If $Y$ is restricted and $\psup{Y}$ is the trivial deformation modulo $h^p$,
	then the $\GG$-torsor $\qcoords$ of quantized coordinates on $X'\series{h}$
	restricts to a $\GG_J$ torsor on $\psup{Y}$.
\end{proposition}	
\begin{proof}
	To show $\qcoords$ reduces to a $\GG_J$-torsor, we must exhibit a section of $\qcoords/\GG_J$ over $\psup{Y}$.
	Since $\GG_J = stab(J)$, a section of $\qcoords/\GG_J$ is a subsheaf of $A_h$ locally isomorphic to $J$.
	Thus, it suffices to prove that the ideal $I_Y \subseteq \OO_h|_{\psup{Y}}$ of $Y$ is locally isomorphic to $J$.
	This is done by our tubular neighborhood theorem, Theorem \ref{theorem: tubular-neighborhood}.
\end{proof}

Denote by $\qcoords_J$ the $\GG_J$-torsor on $\psup{Y}$ which is the restriction of $\qcoords|_{\psup{Y}}$ corresponding to the restricted Lagrangian subvariety $Y$.

Recall from Lemma \ref{lemma: unique-local-quantization} that $A_h/J$ has a unique quantization $M_h = A_h/A_h(y_1,\ldots, y_n)$.
\begin{definition}
	Let $Aut(A_h,M_h)$ denote the $k\series{h}$-group scheme of restricted quantized automorphisms 
	of $A_h$ equipped with a compatible automorphism of $M_h$.
\end{definition}

There is an embedding $\mathbb G_m \to Aut(A_h,M_h)$ sending $r \in \mathbb G_m(R)$ to the identity on $A_R$ and multiplication by $r$ on $M_R$.
There is also a natural map $Aut(A_h,M_h) \to \GG$ given by forgetting the action on $M_h$.
The image of $Aut(A_h,M_h) \to \GG$ is contained in $\GG_J$ since $J$ is the annihilator of $M_h/hM_h$.

\begin{proposition} \label{prop: skolem-noether} 
	If $R$ is a flat $k\series{h}$-algebra, then $\GG_J(R)$ is contained in the image of $Aut(A_h,M_h)$ under the map $Aut(A_h,M_h) \to \GG_J$.
\end{proposition}
\begin{proof}
	Let $\tilde A$ be the subalgebra 
	of $A_h(h^{-1})$ generated over $A_h$ by $h^{-1}J$:
	\[ \tilde A = k\series{h}\langle x_1,\ldots, x_n, h^{-1}y_1,\ldots, h^{-1}y_n\rangle \subseteq A_h(h^{-1}).\]
	The algebra $\tilde A$ is the Weyl algebra modulo the $p$th powers of its generators. It is an Azumaya algebra.
	Since $y_i$ acts on $M_h$ by $h\partial_{x_i}$, the action of $A$ on $M_h$ extends to an action of $\tilde A$. Under this action, $M_h$ is a splitting bundle for $\tilde A$. 
	
	If $R$ is a flat $k\series{h}$-algebra and $\varphi \in \GG_J(R)$, 
	setting $\tilde\varphi(h^{-1}y_i) = h^{-1}\tilde\varphi(y_i)$ 
	induces a unique extension of $\varphi$ to 
	\[ \tilde\varphi: \tilde A \tensor_{k\series{h}} R \to \tilde A \tensor_{k\series{h}} R.\]
	By the Skolem-Noether theorem, the automorphism $\tilde \varphi$ 
	is locally induced by an automorphism of $M_h$,
	which is compatible with $\varphi$ by definition of $\tilde \varphi$.
	Hence $\varphi$ lies in the image of $Aut(A_h,M_h)$.
\end{proof}

By Proposition \ref{prop: skolem-noether} and Lemma \ref{lemma: unique-local-quantization}, the map 
\begin{equation*}
	Aut(A_h,M_h)/\mathbb G_m \to \GG_J
\end{equation*}
is an isomorphism on flat $k\series{h}$-algebra points.
Since $\psup{Y}$ is flat over $k\series{h}$,
we may consider $\qcoords_J$ as a torsor over $Aut(A_h,M_h)/\mathbb G_m$.

\begin{definition} 
	Let $\obstr{\psup{Y}} \in H^2(\psup{Y}, \mathbb G_m)$ be the obstruction to lifting the torsor $\qcoords_J$ to a $Aut(A_h,M_h)$-torsor along 
	\begin{equation}
		\label{eq: aut-of-local-quantization}
		1 \to \mathbb G_m \to Aut(A_h,M_h) \to Aut(A_h,M_h)/\mathbb G_m \to 1.
	\end{equation}
\end{definition}

\begin{proof}[Proof of Theorem \ref{theorem: existence-of-quantization}]
	By Lemma \ref{lemma: unique-local-quantization}, a quantization of a line bundle $\LL$ on $Y$ with $p$-support $\psup{Y}$
	will be locally isomorphic to $M_h$ over $\psup{Y}$.
	Hence, a quantization of $\LL$ induces a lift of $\qcoords_J$ to an $Aut(A_h,M_h)$-torsor of local isomorphisms with $M_h$.
	Conversely, given such a lift $\tilde{P_J}$ to an $Aut(A_h,M_h)$-torsor, 
	the associated bundle of $M_h$ will be a quantization of a line bundle on $Y$.
	Hence, the obstruction to the existence of a quantization of a line bundle on $Y$ with $p$-support $\psup{Y}$ is the same as the obstruction to the existence of a lift of $\qcoords_J$ to an $Aut(A_h,M_h)$-torsor,
	which is $\obstr{\psup{Y}}$ by definition.
\end{proof}

\begin{remark}
	The presence of the hypothesis of flat $k\series{h}$-algebras in this §\ref{ss: torsors-and-quantization} is necessary, as the local structure of deformations is more complicated in the presence of $h$-torsion. 
	In particular, if we consider a deformation of the $p$-support modulo $h^{n+1}$,
	the associated $\GG_J$-torsor is only guaranteed to locally lift to $Aut(A_h,M_h)$ modulo $h^{n-1}$.
	This corresponds to the observation in \cite[§6.3]{bgkp} that cohomology vanishing conditions up to degree $n$ only imply the existence of a deformation to order $n-1$.
\end{remark}

\subsection{Comparison to the class of Bogdanova and Vologodsky} \label{ss: bv}

In this section, we analyze the obstruction $\obstr{\psup{Y}}$ in the case that $\psup{Y} = Y'\series{h}$ by comparing it with the extensions constructed by Bogdanova and Vologodsky in \cite{bv20}.
Their motivation was as follows: inverting $h$ in a Frobenius-constant quantization gives an Azumaya algebra $\OO_h(h^{-1})$ over $X'\lseries{h}$.
However, $\OO_h$ is not Azumaya on $X'\series{h}$.
Bogdanova and Vologodsky show that a correction by a certain reduction of differential operators on $X$ extends to an Azumaya algebra on $X'\series{h}$ \cite[Theorem 1]{bv20}.
We recall their construction below.

Let $A_h^\flat$ be the reduced Weyl algebra on $4n$ variables
$x_i, y_i, \partial_{x_i},\partial_{y_i}$, 
where $x_i$ is dual to $\partial_{x_i}$ and $y_i$ is dual to $\partial_{y_i}$.
We have a canonical inclusion $A_0 \to A_h^\flat$
which sends $x_i \mapsto x_i$ and $y_i \mapsto y_i$.
We also have a map $Der_k(A_0) \to A_h^\flat$ which sends $\partial/\partial x_i$ to $\partial_{x_i}$ and $\partial/\partial y_i$ to $\partial_{y_i}$.
These inclusions make $A_h^\flat$ a quotient of crystalline differential operators $D_{A_0/k,h} \to A_h^\flat$, with kernel generated by $(\partial_{x_1}^p,\ldots, \partial_{y_n}^p)$, cutting out the zero section of the Frobenius-twisted cotangent bundle of $A_0$.

\begin{definition}
	Let $\GG^\flat$ denote the group of restricted quantized automorphisms of $A_h^\flat$.
	Let $G^\flat$ denote the restriction of scalars of $\GG^\flat$ along $k \to k\series{h}$, and $G$ the restriction of scalars of $\GG$ along $k \to k\series{h}$. 
\end{definition}

The quotient $D_{A_0/k,h} \to A_h^\flat$ induces a map $\psi_{can}: G \to G^\flat.$
Concretely, $\psi_{can}(g)$ acts on $A_0 = A_h/hA_h$ by the reduction of $g$ mod $h$,
while it acts on $Der_k(A_0)$ by the induced action on derivations. 
We may also view $A_h^\flat$ as the central reduction of crystalline differential operators along the graph of $\eta = \sum_i y_i dx_i$,
inducing a different $G$-action $\psi: G \to G^\flat$. 
Concretely, $\psi(g) = \varphi_{g^\ast \eta - \eta} \circ \psi_{can}(g)$,
where for an exact $1$-form $\alpha$, 
$\varphi_{\alpha}(f) = f$ for $f \in A_0$ and $\varphi_{\alpha}(\partial) = \partial + \alpha(\partial)$
for $\partial \in Der_k(A_0)$. See \cite[(3.3)]{bv20}.

\begin{remark}
	The maps $\psi_{can}$ and $\psi$ are defined over $k$, not $k\series{h}$,
	and are not induced by morphisms $\GG \to \GG^\flat$ over $k\series{h}$.
\end{remark}

\begin{definition}
	Let $J^\flat = (h,y_1,\ldots, y_n,\partial_{x_1},\ldots, \partial_{x_n}) \subseteq A_h^\flat$.
\end{definition}

Let $G_J$ denote the stabilizer of the ideal $J = (h,y_1,\ldots, y_n)$ in $G$.

\begin{remark}
	Although $\GG_J$ is not representable by a scheme over $k\series{h}$,
	its restriction of scalars $G_J$ is a $k$-subscheme of $G$.
\end{remark}

\begin{lemma}
	The ideal $J^\flat$ is stable under $\psi_{can}(G_J)$ and $\psi(G_J)$.
\end{lemma}
\begin{proof}
	It is clear that $\psi_{can}(G_J)$ and $\psi(G_j)$ preserve $A_h^\flat J$.
	The normalizer $N(J) \subseteq Der_k(A_0)$ of $J$ satisfies
	\[ N(J) = A_0\{ \partial_{x_1},\ldots, \partial_{x_n}\} + J\cdot Der_k(A_0),\]
	since if $\partial = \sum_i f_i \partial_{x_i} + g_i \partial_{y_i}$ normalizes $J$,
	then $\partial(y_i) = g_i \in J$.
	Hence $G_J$ takes $\{\partial_{x_1},\ldots, \partial_{x_n}\}$ into 
	$N(J) \subseteq J^\flat$,
	and thus $J^\flat$ is stable under $\psi_{can}(G_J)$.
	Finally, we compute that $\eta(\partial_{x_i}) = y_i \in J$ for all $i$,
	and hence for any automorphism $g$ in $G_J$,  $g^\ast\eta (\partial_{x_i}) \in J$ also.
	Thus, $J^\flat$ is stable under $\psi(G_J)$.
\end{proof}

Lemma \ref{lemma: unique-local-quantization} shows that $A^\flat_h/J^\flat$ has a unique quantization $M^\flat_h$.
It is a splitting bundle for the subalgebra $A^\flat_h \langle h^{-1}J^\flat\rangle \subseteq A^\flat_h(h^{-1})$ generated over $A^\flat_h$ by $h^{-1}J^\flat$, as in Proposition \ref{prop: skolem-noether}.
Indeed, Proposition \ref{prop: skolem-noether} shows that we have an exact sequence
\begin{equation}
	1 \to L^+\mathbb G_m \to \Res^{k\series{h}}_k Aut(A^\flat_h, M^\flat_h) \to G^\flat_{J^\flat} \to 1,
\end{equation}
where here and below $\Res$ denotes restriction of scalars.
Pulling back along $\psi: G_J \to G^\flat_{J^\flat}$ yields another extension of $G_J$.

Let $L\mathbb G_m$ and $L^+\mathbb G_m$ be the restriction of scalars of $\mathbb G_m$ to $k$ along $k\to k\lseries{h}$ and $k\to k\series{h}$, respectively.
Since $M_h(h^{-1})$ is a splitting bundle for the Azumaya algebra $A_h(h^{-1})$, we have an extension
\begin{equation} \label{eq: loop-splitting-extension}
	 1 \to L \mathbb G_m \to \Res_k^{k\lseries{h}}GL(M_h(h^{-1}))\times_{\Res_k^{k\lseries{h}}Aut(A_h(h^{-1}))} G \to G \to 1,
\end{equation}
and similarly for $A^\flat_h$. 
The main theorem of Bogdanova and Vologodsky is that there is a lattice $\Lambda \subseteq Hom(M_h^\flat, M_h)(h^{-1})$ 
such that $End(\Lambda)$ is stable under the action of $G$ \cite[§3.3]{bv20}. 
While they use an arbitrary splitting bundle for $A_h^\flat$ to construct such a lattice,
we specifically use $M^\flat_h$ associated to $J^\flat$ to facilitate the comparison with the obstruction to quantizing a line bundle on $Y$.

\begin{definition}
	The Bogdanova-Vologodsky class $[\OO_h^\sharp] \in \Het^2(X', L^+\mathbb G_m)$ is the Brauer class of the Azumaya algebra $End(\Lambda) \times_G P$, the associated bundle of $End(\Lambda)$ along the torsor of quantized coordinates $\qcoords$.
\end{definition}

\begin{remark}
	Bogdanova and Vologodsky only showed that there exists an algebra $\OO_h^\sharp$ satisfying certain properties.
	However, Lemma \ref{lemma: unique-loop-reductions} shows that 
	there is a unique reduction of \eqref{eq: loop-splitting-extension} to $L^+\mathbb G_m$, so that the Azumaya algebra $End(\Lambda)$ on the classifying stack $BG$ is unique up to Morita equivalence.
	Hence it makes sense to speak of \emph{the} class of $[\OO_h^\sharp]$.
\end{remark}

We now have three extensions of $G_J$ by $L^+\mathbb G_m$ associated to $M_h, M_h^\flat,$ and $Hom(M_h^\flat, M_h)$, as well as the extension of $G$ by $L^+\mathbb G_m$ associated to $\Lambda$. Proposition \ref{prop: compare-projective-extensions} and Lemma \ref{lemma: unique-loop-reductions} below allow us to compare the obstructions to lifting a $G_J$-torsor along those extensions.

We need terminology for the statement of Proposition \ref{prop: compare-projective-extensions}. 
Given a free $k\series{h}$-module of finite rank $V$,
let $L^+GL(V)$ and $L^+PGL(V)$ denote the restriction of scalars of $GL(V)$ and $PGL(V)$ along $k \to k\series{h}$.
\begin{proposition}\label{prop: compare-projective-extensions}
	Let $H/k$ be a group scheme.
	Given a homomorphism $H \to L^+PGL(V)$ for $V$ a $k\series{h}$-module of finite rank,
	let $[V]$ denote the class of the extension
	\[ 1 \to L^+\mathbb{G}_m \to L^+GL(V) \times_{L^+PGL(V)} G \to G \to 1.\]
	Given such homomorphisms $H \to L^+PGL(V_i)$ for $i=1,2$,
	\begin{enumerate}
		\item $[V_1^*] = [V_1]^{-1}$;
		\item $[V_1 \tensor_{k\series{h}} V_2] = [V_1][V_2]$;
		\item $[Hom_{k\series{h}}(V_1,V_2)] = [V_1]^{-1}[V_2]$.
	\end{enumerate}
\end{proposition}

We can apply this proposition to the projective representation $[M_h]$ from \eqref{eq: aut-of-local-quantization},
along with the corresponding representation $[M^\flat_h]$ pulled back via $\psi$.

Bogdanova and Vologodsky prove the existence of a reduction of \eqref{eq: loop-splitting-extension} from $L\mathbb G_m$ to $L^+\mathbb G_m$. 
The following Lemma shows such reductions are unique:

\begin{lemma} \label{lemma: unique-loop-reductions}
	Let $H$ be an affine group scheme over $k$. Let
	\[ 1 \to L\mathbb G_m \to \tilde K \to H \to 1\]
	be a central extension of $H$ by $L\mathbb G_m$.
	If this extension reduces to an extension by $L^+\mathbb G_m$, then that reduction is unique.
\end{lemma}	
\begin{proof}
	Suppose that $1 \to L^+\mathbb G_m \to K_i \to H \to 1$ are two such reductions for $i=1,2$.
	We have an isomorphism $\varphi: K_1 \times_{L^+\mathbb G_m} L\mathbb G_m\to K_2 \times_{L^+\mathbb G_m} L\mathbb G_m$
	over $H$.
	Then $\varphi$ restricts to an isomorphism $K_1 \to K_2$ over $H$ if and only if the subquotient map
	\[ \overline\varphi: H = K_1/L^+\mathbb G_m\to L\mathbb G_m /L^+\mathbb G_m = Gr_{\mathbb G_m}\]
	is trivial.
	However, every group homomorphism from the affine group scheme $H$ to  $Gr_{\mathbb G_m}$ is trivial, by \cite[Corollary 6.3]{bv20}.
\end{proof}

\begin{proof}[Proof of Theorem \ref{theorem: comparison-to-bv}] 
	Let $\oflat$ and $\mathfrak{o}^\#$ denote the obstructions to lifting $\qcoords_J$ to a torsor over
	the extensions of $G_J$ associated to $[M_h^\flat]$ and $[Hom(M_h^\flat, M_h)]$, respectively.
	By Proposition \ref{prop: compare-projective-extensions},
	\[ \obstr{Y'\series{h}} = \mathfrak{o}^\# \cdot \oflat.\]
	Now $\Lambda$ and $Hom(M_h^\flat, M_h)$ are both lattices in $Hom(M_h^\flat, M_h)(h^{-1})$ whose endomorphisms are stable under $G_J$.
	Lemma \ref{lemma: unique-loop-reductions} implies that the associated group extensions are isomorphic, and thus $\mathfrak{o}^\#$ is the obstruction to lifting $\qcoords_J$ to a torsor over $GL(\Lambda)\times_{PGL(\Lambda)} G_J$.
	Since $End(\Lambda)$ is also $G$-stable, this obstruction is the restriction of the obstruction to lifting the full torsor of quantized coordaintes $\qcoords$ to a torsor over $GL(\Lambda)\times_{PGL(\Lambda)} G$,
	which is $[\OO_h^\sharp]$.
	Hence 
	\[\mathfrak{o}^\# = [\OO_h^\sharp]|_{Y'}.\]

	Now $\oflat$ is the obstruction for the torsor $\qcoords_J$ to admit a lift along 
	the pullback of 
	\[ 1 \to L^+\mathbb G_m \to \Res_k^{k\series{h}} Aut(A^\flat_h, M^\flat_h) \to G^\flat_{J^\flat} \]
	via $\psi: G_J \to G^\flat_{J^\flat}$.
	This torsor has a lift if and only if there is a module over $D_{X,[\eta],h}|_{Y'\series{h}}$ which is locally isomorphic to 
	\[M^\flat_h = A^\flat_h/A^\flat_h(y_1,\ldots,y_n,\partial_{x_1},\ldots,\partial_{x_n}).\]
	Since $Y$ is restricted, $[\eta]|_Y = 0$,
	so $D_{X,[\eta],h}|_{Y'} = D_{X,0,h}|_{Y'}$,
	and the desired module is $D_{X,0,h}|_{Y'}/D_{X,0,h}|_{Y'}(I_Y + T_Y).$ 
	Thus, $\oflat = 1$.
\end{proof}

\section{Restricted Chern classes} \label{section: restricted-chern-classes}

In characteristic zero, given that a quantization of a Lagrangian subvariety $Y$ exists, 
a particular line bundle $\LL$ on $Y$ admits a quantization if and only if the Atiyah class of $\LL$
satisfies a certain equation. For instance, if the quantization is self-dual, this equation states
that $\LL^{\tensor -2} \tensor K_Y$ admits a flat algebraic connection \cite{bgkp}. 

In positive characteristic, flat connections admit a nontrivial invariant, their $p$-curvature. 
The positive-characteristic version of the Atiyah class is the obstruction to the existence of a flat connection with $p$-curvature zero.
This will allow us to classify which line bundles on $Y$ admit quantizations.

\subsection{Restricted Atiyah algebras}

Lie algebroids were introduced by Rinehart in \cite{rinehart63}. Below, we define what it means for a Lie algebroid to have a restricted structure.
\begin{definition}
	Let $Z/k$ be a smooth variety over a field $k$ of positive characteristic.
	A \emph{restricted Lie algebroid} on $Z$ is a Lie algebroid $\tau: \Aa \to T_Z$, locally free as an $\OO_Z$-module, equipped with a restricted operation $x \mapsto x^\pop$ such that:
	\begin{enumerate}
		\item $-^\pop$ makes $\Aa$ into a restricted Lie algebra,
			and the anchor map is a restricted Lie homomorphism.
		\item for $f$ a section of $\OO_Z$ and $x$ a section of $\Aa$,
			\[ (fx)^\pop = f^p x^\pop + (\tau(fx))^{p-1}(f)x.\]
	\end{enumerate}
\end{definition}
The definition is essentially due to Hochschild, who considered such algebroids in connection with the Galois theory of inseparable extensions \cite{hochschild55}. See also \cite[§3.1]{rumynin00}.
\begin{definition}
	A \emph{restricted Atiyah algebra} on $Z$ is a restricted Lie algebroid on $Z$ of the form
	\[ 0 \to \OO_Z \to \Aa \to T_Z \to 0\]
	such that $[x,f] = \tau(x)(f)$ for all sections $f$ of $\OO_Z$ and $x$ of $T_Z$, 
	and for every local section $f$ of $\OO_Z$, $f^\pop = f^p$.
\end{definition}

\begin{example}
	If $\LL$ is a line bundle on $Z$, then $Diff^{\leq 1}(\LL)$, 
	the sheaf of differential operators on $\LL$ of order at most one,
	is a restricted Atiyah algebra on $Z$ with $p$-operation $x^\pop = x^p$ as an operator on $\LL$. The key identity for restricted Lie algebroids is satisfied because of Hochschild's identity \cite[Lemma 1]{hochschild55}.
\end{example}

Local splittings of the anchor map $\Aa \to T_Z$ are called connections on $\Aa$. In the case when $\Aa = Diff^{\leq 1}(\LL)$, these are equivalent to connections on $\LL$. In the algebraic category in characteristic zero, it may be that every connection on an Atiyah algebra has nonzero curvature. 
However, restricted Atiyah algebras in positive characteristic enjoy the special property that there are always local connections with zero curvature. This is the content of Lemma \ref{lemma: locally-split} below.

\begin{lemma} \label{lemma: locally-split}
	If $\Aa$ is a restricted Atiyah algebra on $Z$,
	then the anchor map $\tau: \Aa \to T_Z$ is locally split as a map of Lie algebras.
\end{lemma}
\begin{proof}
	The question is local, so we may assume that there is a $\OO_Z$-linear section $\sigma: T_Z \to \Aa$ of the anchor map $\tau$.
	Let $\beta \in \Omega^2_Z$ be the curvature of this splitting, that is,
	\[ \beta(x,y) = [\sigma x, \sigma y] - \sigma[x,y].\]
	Sections are a torsor over $\Omega^1_Z$, and the curvature of the splitting $\sigma  +\alpha$ for $\alpha$ a 1-form is $\beta + d\alpha$.
	Hence, we wish to show that $\beta$ is locally exact.
	The Jacobi identity for $\Aa$ shows that $d\beta = 0$.
	Now we compute the image $C(\beta)$ of $\beta$ under the Cartier operator.
	By Lemma \ref{lemma: cartier-and-restricted-derivative}, the Cartier operator satisfies 
	\[ C(i^\pop_x \beta) = i_{x'} C(\beta)\]
	for all vector fields $x$.
	For any 1-form $\alpha$ and vector fields $x$ and $y$, we have by induction
	\[ (L_x^i \alpha)(y)= \sum_{j=0}^i (-1)^j \binom{i}{j} x^{i-j} \alpha( ad^j x(y)),\]
	and thus by definition of $\beta$,
	\begin{align*}
		(L_x^{p-1} i_x \beta) (y) &= \sum_{j=0}^{p-1} x^{p-1-j} \beta(x, ad^j x(y)) \\
			&= \sum_{j=0}^{p-1} ad^{p-1-j}\sigma(x)\left( [\sigma x, \sigma(ad^j x(y))] - \sigma[x,ad^j x(y)]\right) \\
					&= \sum_{j=0}^{p-1} ad^{p-j} \sigma(x)(\sigma(ad^j x(y))) - ad^{p-j-1}\sigma(x)(\sigma(ad^{j+1} x(y))) \\
					&= [\sigma(x)^\pop, \sigma y] - \sigma([x^\pop,y]).
	\end{align*}
	Hence
	\begin{align*}
		i^\pop_x \beta (y) &= [\sigma(x^\pop), \sigma(y)] - \sigma[ x^\pop,y]
		- [\sigma(x)^\pop, \sigma(y)] + \sigma[x^\pop, y] \\
			&= [ \sigma(x^\pop) - \sigma(x)^\pop, \sigma y].
	\end{align*}
	Since $\tau$ is a map of restricted Lie algebras, 
	\[\tau(\sigma(x^\pop) - \sigma(x)^\pop) = 0,\]
	so that $\sigma(x^\pop) - \sigma(x)^\pop$ is a section of $\OO_Z$.
	Hence $i_x^\pop \beta$ is exact with primitive $-(\sigma(x^\pop) - \sigma(x)^\pop)$,
	so $i_x C(\beta) = 0$.
	As $x$ was arbitrary, $C(\beta) = 0$, so by the Cartier isomorphism, $\beta$ is locally exact.
\end{proof}	

While restricted Atiyah algebras have local integrable connections, these connections may have nonzero $p$-curvature.
The classification of restricted Atiyah algebras comes down to taking $p$-curvature into account.

\begin{theorem} \label{theorem: classify-atiyah-algebras}
	Let $Z/k$ be a smooth variety and $C$ be the Cartier operator. 
	Let $\aclasscomplex{Z}$ be the complex
	\[ 
		\aclasscomplex{Z}= \left\{\begin{diagram}
			\Omega^1_{Z,cl} & \rTo^{\alpha \mapsto \alpha'-C(\alpha)} & \Omega^1_{Z'}
		\end{diagram}\right\}.
	\]
	Then restricted Atiyah algebras on $Z$ up to isomorphism are in bijection with $H^1(Z,\aclasscomplex{Z})$.
\end{theorem}	
\begin{proof}
	We describe how to send a restricted Atiyah algebra $\Aa$ to a class $[\Aa] \in H^1(\aclasscomplex{Z})$.
	The anchor map $\tau: \Aa \to T_Z$ is locally split as a map of Lie algebras.
	Let $\{U_i\}$ be a Zariski open cover with splittings $\sigma_i$ of $\tau$ over $U_i$,
	and let $\alpha_{ij} = \sigma_i - \sigma_j \in \Omega^1_{U_{ij}}$.
	Since the maps $\sigma_i$ respect the bracket, $d\alpha_{ij} = 0$.

	For all local vector fields $x$, $\sigma_i(x)^\pop - \sigma_i(x^\pop)$
	is central. Hence there exists $\gamma_i \in \Omega^1_{U_i}$ such that
	\[ \gamma_i(x)^p = \sigma_i(x)^\pop - \sigma_i(x^\pop).\]
	The form $\gamma_i$ determines the restricted structure of $\Aa$ over $U_i$,
	and must satisfy the compatibility
	\begin{align*}
		(\gamma_i - \gamma_j)^p(x) &= \sigma_i(x)^\pop - \sigma_j(x)^\pop - \sigma_i(x^\pop) + \sigma_j(x^\pop) \\
		&= (\sigma_j(x) + \alpha_{ij}(x))^\pop - \sigma_j(x)^\pop - \alpha_{ij}(x^\pop) \\
		&= \alpha_{ij}(x)^p + x^{p-1}\alpha_{ij}(x) - \alpha_{ij}(x^\pop),
	\end{align*}
	which is equivalent to $\gamma_i - \gamma_j = \alpha_{ij}' - C(\alpha_{ij})$.
	This shows that $\{ \alpha_{ij}, \gamma_i\}$ is a \v{C}ech 1-cocycle for $\aclasscomplex{Z}$ 
	and thus defines a class in $H^1(\aclasscomplex{Z})$.
	A refinement of the cover or a change of splittings does not change the class,
	and thus there is a well-defined class $[\Aa] \in H^1(\aclasscomplex{Z})$.
	
	The class $[\Aa]$ is exactly the obstruction to the existence of a global splitting $\sigma: T_Z \to \Aa$ respecting the restricted structure, and hence the map sending an Atiyah algebra $\Aa$ to $[\Aa]$ is injective.
	The map is surjective since the data of a \v{C}ech class $\{\alpha_{ij}, \gamma_i\}$ 
	allows the Atiyah algebras $\OO_{U_i} \oplus T_{U_i}$ with restricted operation
	\[ (f,x)^\pop = (f^p + \gamma_i(x)^p, x^\pop)\]
	to be glued to a restricted Atiyah algebra on $Z$.
\end{proof}	

\begin{remark}
	Theorem \ref{theorem: classify-atiyah-algebras} is similar to the main theorem of \cite{hochschild54}, 
	where restricted extensions of restricted Lie algebras are classified as the ordinary Lie algebra extensions
	on which a certain Cartier-type invariant vanishes.
	This Cartier-type invariant is exactly the $d_2$-differential in the Friedlander-Parshall spectral sequence
	relating non-restricted to restricted Lie algebra cohomology \cite{friedlander86}.
\end{remark}

\begin{remark}
	Atiyah algebras without restricted structure are classified by $H^1$ of the complex $\Omega_Z^{\geq 1} = (\Omega^1_Z \to \Omega^2_Z \to \cdots)$. 
	This is well-known in characteristic zero \cite[Lemma 2.1.6]{bb}.
	The natural maps 
	\[ \aclasscomplex{Z} \to \Omega^1_Z \to \Omega_Z^{\geq 1}\]
	induce a map $H^1(\aclasscomplex{z}) \to H^1(\Omega^{\geq 1}_Z)$ which corresponds to forgetting the restricted structure.
\end{remark}

For smooth $Z/k$, there is a well-known exact sequence of étale sheaves
\begin{equation*}
	\begin{diagram}
		1 \to \OO_{Z'}^\times	&\rTo	&\OO_Z^\times	&\rTo^{dlog}	
			&\Omega^1_{Z,cl}	&\rTo^{\alpha \mapsto \alpha'-C(\alpha)}	&\Omega^1_{Z'}  \to 0,
	\end{diagram}
\end{equation*}
cf.\ \cite[Proposition 4.14]{milne-etale}. 
The sequence shows that $\aclasscomplex{Z}$ is quasiisomorphic in étale topology to the sheaf of logarithmic forms 
$\Omega^1_{Z,log}$.
Taking étale cohomology yields the sequence
\begin{equation} \label{chern-cohomology}
	\cdots \to Pic(Z') \to Pic(Z) \to H^1(\aclasscomplex{Z}) \to \Het^2(Z', \mathbb G_m) \to \Het^2(Z, \mathbb G_m) \to \cdots
\end{equation}
The map $Pic(Z) \to H^1(\aclasscomplex{Z})$ is induced by sending a line bundle $\LL$ 
to $[Diff^{\leq 1}(\LL)]$.

\begin{remark}
	The sequence \eqref{chern-cohomology} shows that Brauer classes on $Z'$ which split over the relative Frobenius $Z \to Z'$ 
	are in bijection with $H^1(\aclasscomplex{Z})/Pic(Z)$, that is, restricted Atiyah algebras modulo those of the form $Diff^{\leq 1} (\LL)$ for $\LL \in Pic(Z)$.
	This is analogous to the main theorem of \cite{hochschild55}, where $Z$ replaces a height 1 inseparable extension $L/k$. 
	In that case, $Pic(L)$ is trivial.
\end{remark}

\begin{definition}
	The map $Pic(Z) \to H^1(\aclasscomplex{Z})$ 
	sending $\LL \mapsto [Diff^{\leq 1}(\LL)]$
	is the \emph{restricted Chern class} $c_r: Pic(Z) \to H^1(\aclasscomplex{Z})$.
\end{definition}

\subsection{Atiyah algebras and quantizations}

In characteristic zero, there is an Atiyah algebra associated to a quantization and a Lagrangian subvariety 
which controls the Chern class of quantizable line bundles \cite[§5.3]{bgkp}.
In the case when $Y \subseteq X$ is restricted in positive characteristic,
we will construct a restricted Atiyah algebra.

\begin{proposition}
	Let $I \subseteq \OO_h$ be the preimage under $\OO_h \to \OO_X$ of the ideal of a restricted Lagrangian subvariety $Y$.
	Then $I/I^2$ is a restricted Atiyah algebra on $\OO_Y$.
\end{proposition}
\begin{proof}
	Recall $J = (h,y_1,\ldots, y_n) \subseteq A_h$.
	There is a short exact sequence of $\GG_J$-modules
	\[ 0 \to A_h/J \to J/J^2 \to J/(h+J^2) \to 0,\]
	where the first arrow is multiplying by $h$.
	By Proposition \ref{proposition: restriction-to-GJ},
	if $Y \subseteq X$ is restricted Lagrangian,
	then the $\GG$-torsor of quantized coordinates
	restricts to a $\GG_J$-torsor over $\tilde{Y'} = Y'\series{h}$
	such that the associated bundle of the $\GG_J$-representation $J$ is $I$.
	Hence we have a short exact sequence of sheaves
	\[ 0 \to \OO_Y \to I/I^2 \to I/(h + I^2) \to 0.\]
	As $\OO_X = \OO_h /h\OO_h$, the sheaf $I/(h+I^2)$ is the conormal bundle 
	of $Y$ in $X$, which is identified under the symplectic form with the tangent bundle of $Y$.
	The induced map $\tau: I/I^2 \to T_Y$ is $\tau: f \mapsto H_f$,
	and makes $I/I^2$ into a Lie algebroid under the Poisson bracket.

	Now we show that the restricted operation $-^\pop$ descends to an operation on $I/I^2$.
	Let $Q(x,y)$ be the free quantized algebra on $x$ and $y$.
	This is the Rees algebra of the tensor algebra $T(x,y)$ with respect to the PBW filtration from the free Lie algebra on $x$ and $y$. 
	It is homogeneous with respect to three different gradings: the grading where $|h| =1$ and $|\ell| = 1$ for any $\ell$ in the free Lie algebra on $x$ and $y$,
	grading with respect to $x$, and grading with respect to $y$.
	 
	By \cite[§1.2]{bk}, there is $P(x,y) \in Q(x,y)$ defined by
	\[h^{p-1}P(x,y) = (xy)^p - x^py^p. \]
	Its main property is that for $x$ and $y$ in any Frobenius-constant quantization,
	\[(xy)^\pop = x^py^\pop + x^\pop y^p - h^{p-1}x^\pop y^\pop + P(x,y).\]
	With respect to our three gradings, $P$ has degree $(p+1,p,p)$.
	
	By the PBW theorem, $Q(x,y)$ has a homogeneous $k$-basis where every element is of the form
	$h^k \ell_1\cdots \ell_m$
	where $\ell_j$ are Lie monomials in $x$ and $y$.
	Hence we may express $P$ as a sum of such monomials where each monomial is of degree $(p+1,p,p)$.
	But every such monomial is either $h\{x,y\}^p$ or contains $hx$, $hy$, or both $x$ and $y$ as factors.
	Since $I$ is two-sided and $\{I,I\} + I^\pop \subseteq I$, we conclude that if $a,b \in I$, then
	\[ (ab)^\pop = a^p b^\pop + a^\pop b^p - h^{p-1}a^\pop b^\pop + P(a,b) \in I^2.\]
	Hence $I^2$ is stable under the restricted operation. 
	Then for $a \in I$ and $b \in I^2$,
	\[ (a+b)^\pop = a^\pop + b^\pop + L(a,b)\]
	where $L$ is Jacobson's Lie polynomial, so since $I^2 \subseteq I$ is a Lie ideal,
	we conclude that $-^\pop$ descends to $I/I^2$,
	making $I/I^2$ into a restricted Lie algebra.
	Since $H_f^\pop = H_{f^\pop}$ for all $f \in \OO_X$,
	the anchor map $\tau: I/I^2 \to T_Y$ respects the restricted structure.

	Now we check the second axiom of restricted Lie algebroids.
	For $f$ a section of $\OO_h$ and $x$ a section of $I$,
	\[ (fx)^\pop = f^p x^\pop + f^\pop x^p - h^{p-1} f^\pop x^\pop + P(x,y).\]
	For $x \in I$, $x^p$ and $h^{p-1}\OO_h$ are both contained in $I^2$. 

	Now we must use Hochschild's identity on products of the form $(yx)^p$.
	Let $T$ be the tensor algebra on $x$ and $y$,
	and $R = ([ad_x^i y, ad_x^j y]\mid i,j \geq 0)$.
	Then the subalgebra of $T/R$ generated by $ad_x^i y$ for $i \geq 0$ is commutative,
	and so Hochschild's identity \cite[Lemma 1]{hochschild55} shows that 
	\[ (yx)^p - y^px^p = ad_{yx}^{p-1}(y)x \in T/R.\]
	Upon taking the Rees algebra, and taking $R' = (\{ad_x^i y, ad_x^j y \} \mid i,j \geq 0)$,
	we obtain 
	\begin{equation} \label{P-minus-hochschild}
		P(x,y) - ad_{yx}^{p-1}(y) x \in R',
	\end{equation}
	where now $ad$ is the adjoint action of the Poisson bracket.
	But \eqref{P-minus-hochschild} is homogeneous of degree $(p+1,p,p)$.
	As $R'$ is a monomial ideal, there must be an expression for $P(x,y) - ad_{yx}^{p-1}(y)x$
\	where each monomial is a multiple of $\{ad_x^i y, ad_x^j y\}$ for some $i, j \geq 0$.
	But for such a monomial to have total degree $p+1$ and $y$-degree $p$,
	it must have two terms with $y$-degree zero, and thus contains $h^2$, $hx$, or two copies of $x$.

	Thus, if $x \in I$, then $P(f,x) - ad_{fx}^{p-1}(f)x \in I^2$.
	Hence for $x \in I/I^2$ and $f \in \OO_Y$,
	\[(fx)^\pop = f^px^\pop + ad_{fx}^{p-1}(f)x\]
	in $I/I^2$, as desired.

	For $f \in \OO_h$, $(hf)^\pop = hf^p$, so $I/I^2$ obeys the final axiom of restricted Atiyah algebras.
\end{proof}

We now aim to compute the Atiyah class of $I/I^2$.
The algebroid $I/I^2$ depends only on the first-order restricted quantization $\OO_h/h^2\OO_h$, and so we expect that it can be expressed in terms of the first-order data of $\OO_h$. We recall Bezrukavnikov and Kaledin's classification of first-order restricted quantizations.

\begin{proposition}\cite[Proposition 1.19]{bk} \label{prop: first-order-torsor}
	First-order restricted quantizations of $X$ are a torsor over $H^1(X', \Omega^1_{X,log})$.
	There is a canonical basepoint for this torsor, the first-order quantization $\OO_h$ such that $\OO_h \cong \OO_{-h}^{op}$,
	inducing a bijection 
	\begin{equation} \label{eq: first-order-quantization-map}
		\rho: \left\{ 	\text{first-order restricted quantizations over }\OO_X\right\} 
		\to \Het^1(X', \Omega^1_{X,log}).
	\end{equation}
\end{proposition}
See \cite[§2.6]{bv20} for another description of this classifying map $\rho$.
In characteristic zero, this canonical basepoint is known as the \emph{canonical quantization} \cite[Definition 1.9]{bk04}.

The following lemma identifies the class of the dual restricted Atiyah algebra, which will be used to calculate $[I/I^2]$.

\begin{lemma} \label{lemma: dual-atiyah-class}
	Given a restricted Atiyah algebra $\Aa$ on $Z$,
	the opposite Atiyah algebra $\Aa^{op}$
	is restricted with the same $p$-operation.
	Further,
	\[ [\Aa] + [\Aa^{op}] = c_r(K_Z),\]
	where $K_Z$ is the canonical bundle.
\end{lemma}

\begin{proof}
	The operation $a \mapsto a^\pop$ is a restricted Lie operation for the opposite Lie algebra, and the anchor map 
	$\tau^{op} = -\tau: \Aa^{op} \to T_Z$
	is compatible with the operation: for $a$ a local section of $\Aa$,
	\[\tau^{op}(a^\pop) = - \tau(a)^\pop = (-\tau(a))^\pop = \tau^{op}(a)^\pop.\]
	The compatibility of opposite scalar multiplication and $a \mapsto a^\pop$ follows from Hochschild's identity. Compute in $\Aa$: for $f$ a local section of $\OO_Z$ and $a$ a local section of $\Aa$,
	\begin{align*}
		(af)^\pop - a^\pop f^p &= (fa + \tau a(f))^\pop - a^\pop f^p \\
			&= (fa)^\pop - f^p a^\pop + (\tau a(f))^p + (f\tau a)^{p-1}(\tau a(f)) \\
			&= (f\tau a)^{p-1}(f)a + (\tau a(f))^p + (f\tau a)^{p-1}(\tau a(f)),
	\end{align*}
	whereas Hochschild's identity in the ring of differential operators gives
	\begin{align*}
		a (f\tau a)^{p-1}(f) &= (f\tau a)^{p-1}(f) a + \tau a(f\tau a)^{p-1}(f) \\
			&= (f\tau a)^{p-1}(f) a + (\tau a f)^{p}(1) \\
			&= (f\tau a)^{p-1}(f) a + (f \tau a + \tau a(f))^p (1) \\
			&= (f \tau a)^{p-1}(f) a + (\tau a(f))^p + (f \tau a)^{p-1}(\tau a(f)).
	\end{align*}
	Comparing our two equations shows that $\Aa^{op}$ is a restricted Lie algebroid. 	
	Now the proof is the same as for Atiyah algebras in characteristic zero; see \cite[§2.4.1]{bb}.
\end{proof}

\begin{lemma} \label{class-of-I/I^2}
	Let $Y\subseteq X$ be a restricted Lagrangian subvariety and $I$ be the ideal of $Y$ in $\OO_h$. Then in $H^1(\aclasscomplex{Y}) = \Het^1(\Omega^1_{Y,log})$,
	\[ [I/I^2] = \rho(O_h)|_Y + \frac{1}{2}c_r(K_Y),\]
	where $\rho$ is the classifying map \eqref{eq: first-order-quantization-map}.
\end{lemma}
\begin{proof}
	Suppose first that $\OO_h/h^2$ is a self-dual first-order quantization. 
	Let $\alpha: \OO_h/h^2 \to \OO_{-h}^{op}/h^2$ be the duality isomorphism.
	Since $\alpha \equiv id \mod h$,
	the following diagram induced by $\alpha$ commutes:
	\[	
		\begin{diagram}
			0&\rTo 	& \OO_Y & \rTo^h & (I/I^2)^{op} & \rTo^{\tau^{op} = -\tau} & T_Y & \rTo & 0 \\
		& 	& \dTo^{id}	& 	& \dTo^{-\alpha} & & \dTo^{id} & & \\
		0&\rTo 	& \OO_Y	& \rTo^h & I/I^2	& 	\rTo^{\tau} & T_y & \rTo & 0 
		\end{diagram}.
	\]
	Since $\alpha$ is an isomorphism of Poisson structures, $-\alpha$ takes the Lie bracket on $(I/I^2)^{op}$ to the bracket on $I/I^2$.
	We conclude that $[I/I^2] = [(I/I^2)^{op}]$, which by Lemma \ref{lemma: dual-atiyah-class} shows that $[I/I^2] = \frac{1}{2} c_r(K_Y)$.

	In general, $\Het^1(\Omega^1_{X,log})$ acts on first-order quantizations via $\rho$ as in \eqref{eq: first-order-quantization-map}, and acts on $\Het^1(\Omega^1_{Y,log}) = H^1(\aclasscomplex{Y})$ via restriction along $Y \to X$. These actions are compatible: twisting by a étale \v{C}ech cocycle $\{\alpha_i\}$ of logarithmic forms induces local automorphisms $f \mapsto f + h \alpha_i(H_f)$ of $\OO_h/h^2$, which induces automorphisms by the same formula on $I/I^2$.
	This proves the formula.
\end{proof}

Now we can classify the restricted Chern classes of quantizable line bundles on $Y$. 
Recall that infinitesimal deformations of an $R$-subscheme $W\subseteq Z$ along a square-zero extension $0 \to I \to R' \to R \to 0$ are a torsor over $H^0(W,N_{Z/W} \tensor_R I)$. Now suppose $\psup{Y}$ is a formal deformation of $Y'\subseteq X'$ which is trivial modulo $h^p$.
Then deformations of $\psup{Y}/h^p= Y'[h]/h^p$ to order $h^p$ have a canonical basepoint: the trivial deformation $Y'[h]/h^{p+1}$. Hence, it makes sense to canonically associate a normal field $\theta \in H^0(Y', N_{Y'/X'})$ 
to the deformation $\psup{Y}/h^{p+1}$ of $\psup{Y}/h^p = Y'[h]/h^p$.

\begin{theorem} \label{theorem: chern-class-of-quantization}
	Suppose $Y \subseteq X$ is a restricted Lagrangian subvariety and 
	$\psup{Y}$ is a formal deformation of $Y'\subseteq X'$ which is trivial modulo $h^p$.
	Let $\theta \in H^0(Y',N_{Y'/X'})$ be the normal field
	defining the infinitesimal deformation $\psup{Y}/h^{p+1}$ of the subscheme $\psup{Y}/h^p = Y'[h]/h^p$ of $X'[h]/h^p$.
	
	If $\LL$ is a line bundle on $Y$ admitting a quantization with $p$-support $\psup{Y}$, then
	\begin{equation} \label{chern-class-equation}
		 c_r(\LL) = \rho(\OO_h)|_Y + \frac{1}{2}c_r(K_Y) - [i_\theta \omega'],
	\end{equation}
	where $K_Y$ is the canonical bundle on $Y$ and $[i_\theta \omega']$ is the Atiyah class on $Y$ corresponding to the 1-form $i_\theta \omega'$ on $Y'$.

	Conversely, 
	if such a quantization exists 
	and the restriction $Pic(\psup{Y}) \to Pic(Y')$ is onto,
	then every line bundle $\LL$ on $Y$ satisfying \eqref{chern-class-equation}
	admits a quantization.
\end{theorem}
\begin{proof}
	Suppose $\LL$ admits a quantization $\LL_h$. 
	Let $I$ be the preimage of the ideal of $Y$ in $\OO_h$.
	The map $\varphi: I/I^2 \to Diff^{\leq 1}(\LL)$ 
	defined by $a \mapsto h^{-1}a: \LL_h/h \to \LL_h/h$
	is a map of nonrestricted Lie algebras: 
	$h^{-1}a$ is a differential operator of order at most one,
	with principal symbol $\tau(a) = H_a$ since $[h^{-1}a, f] = \{a,f\}$
	for $f$ a local section of $\OO_Y$.
	
	Let us compute $\varphi(a)^\pop - \varphi(a^\pop)$.
	We have $\varphi(a)^\pop = (h^{-1}a)^p = h^{-p}a^p$,
	while $\varphi(a^\pop) = h^{-1} a^\pop$. Hence for $a \in I/I^2$,
	\[ \varphi(a)^\pop - \varphi(a^\pop) = h^{-p}a^p - h^{-1}a^\pop = h^{-p}s(a),\]
	where $s$ is the splitting map of the Frobenius-constant quantization, by \eqref{p-operation-definition}.
	By definition of $\theta$, the action of $s(a)$ on $\LL_h/h^{p+1}$
	is exactly $h^p \theta( s(a))$. 
	Thus $\varphi$ is not a map of restricted Lie algebroids.

	Let $\Aa$ be the Atiyah algebra $Diff^{\leq 1}(\LL) + [i_\theta \omega']$.
	It is the same underlying Lie algebroid as $Diff^{\leq 1}(\LL)$, 
	but with $p$-operation $x^{\pop_{\Aa}} =  x^\pop + i_{(\tau x)'} i_\theta \omega'$.
	Let $\varphi^\theta: I/I^2 \to \Aa$ be the map $a \mapsto h^{-1}a$ as above.
	Then
	\[ (\varphi^\theta(a))^\pop = h^{-p}a^p + i_{H_a'}i_\theta \omega' = h^{-p} a^p - i_\theta ds(a) = h^{-p} a^p - \theta(s(a)),\]
	which shows that $\varphi^\theta$ is a map of restricted Atiyah algebras.
	Hence by Lemma \ref{class-of-I/I^2},
	\[ 
		c_r(\LL) = [I/I^2] - [i_\theta \omega'] = \rho(\OO_h)|_Y + \frac{1}{2}c_r(K_Y) - [i_\theta \omega'].
	\]

	Since quantizations of $Y$ with $p$-support $\psup{Y}$ are equivalent to lifts of the $\GG_J$-torsor $\qcoords_J$ to a $Aut(A_h,M_h)$-torsor, quantizations of $Y$ are a torsor over $H^1( \psup{Y}, \mathbb G_m) = Pic(\psup{Y})$. Concretely, the action by $\mathcal F \in Pic(\psup{Y})$ is 
	\[ \mathcal F : \LL_h \mapsto \mathcal F \tensor_{\OO_{\psup{Y}}} \LL_h.\]
	If $Pic(\psup{Y}) \to Pic(Y')$ is onto, then the line bundles $\LL \in Pic(Y)$ admitting quantizations are closed under the action of $Pic(Y')$.
	By Cartier descent (which is encoded in the sequence \eqref{chern-cohomology}), the difference of two line bundles descends to $Y'$
	if and only if the difference carries a flat connection with zero $p$-curvature, that is, the restricted Chern class of the two bundles are equal.
	If a quantization of a line bundle on $Y$ with $p$-support $\psup{Y}$ exists,
	it follows that $\LL \in Pic(Y)$ may be quantized if and only if 
	its restricted Chern class satisfies \eqref{chern-class-equation}.
\end{proof}

\section{Morita equivalence of quantizations}

We give a criterion for Morita equivalence of Frobenius-constant quantizations.
After inverting $h$, a Frobenius-constant quantization becomes an Azumaya algebra, and thus Morita equivalence is determined by a Brauer class.
Surprisingly, a Brauer class also controls Morita equivalence across $h=0$.
This application of the main result is inspired by forthcoming work on categorical quantization in positive characteristic by E. Bogdanova, D. Kubrak, R. Travkin, and V. Vologodsky. 

Let $(X,[\eta])$ be a restricted symplectic variety as above.
The symplectic form $\omega$ satisfies $d[\eta] = \omega$.
Then the symplectic variety $X^- = (X,-\omega)$ is also restricted,
as $-\omega = d[-\eta]$. 
We will give $X^-$ this restricted structure.
\begin{lemma} \label{lemma: diagonal-restricted}
	The diagonal $\Delta_X \subseteq  X^- \times X$ 
	is a restricted Lagrangian subvariety.
\end{lemma}	
\begin{proof}
	The restricted structure on $X^- \times X$ is given by the class 
	\[ 1 \tensor [\eta] - [\eta] \tensor 1 \in H^1(\Omega^{\leq 1}),\]
	which evidently restricts to zero on $\Delta_X$.
\end{proof}

Now suppose $\OO_h^1$ and $\OO_h^2$ are two Frobenius-constant quantizations of $X$. Then $(\OO_h^1)^{op}$ with the same Frobenius splitting map is a quantization of $X^-$, for the restricted structure agrees with $[-\eta]$ by Theorem \ref{theorem: restricted-structure-from-form}.

\begin{lemma}
	For $\OO_h$ a Frobenius-constant quantization of $X$,
	the Bogdanova-Vologodsky class of the quantization $\OO_h^{op}$ of $X^-$ 
	is $[\OO_h^\sharp]^{-1}$.
\end{lemma}
\begin{proof}
	Apply the canonical identification $End(V)^{op} = End(V^*)$ for a vector space $V$ throughout Bogdanova and Vologodsky's construction of $\OO_h^\sharp$, and apply Proposition \ref{prop: compare-projective-extensions}.
\end{proof}

\begin{theorem}
	\label{theorem: morita-equivalence}
	Let $\OO_h^1$ and $\OO_h^2$ be two Frobenius-constant quantizations of a restricted symplectic variety $X$.
	If 
	\[[(\OO_h^1)^\sharp]=  [(\OO_h^2)^\sharp],\] 
	then $\OO_h^1$ and $\OO_h^2$ are Morita equivalent over their centers.
\end{theorem}
\begin{proof}
	Consider the diagonal $\Delta_X \subseteq X^- \times X$.
	It is restricted Lagrangian by Lemma \ref{lemma: diagonal-restricted}.
	The algebra $(\OO_h^1)^{op} \tensor_{k\series{h}} \OO_h^2$ is a Frobenius-constant quantization of $X^- \times X$.
	The hypothesis implies 
	\[ [(\OO_h^2)^\sharp]^{-1} \tensor [(\OO_h^1)^\sharp] \]
	restricts to zero on $\Delta_{X'}\series{h} \subseteq (X'\times X')\series{h}$.
	By Theorems \ref{theorem: existence-of-quantization} and \ref{theorem: comparison-to-bv},
	there is a line bundle $\LL$ on $\Delta_X$ 
	admitting a quantization to $\LL_h$ over $(\OO_h^1)^{op}\tensor \OO_h^2$.
	Then $\LL_h$ is the bimodule inducing a Morita equivalence between $\OO_h^1$ and $\OO_h^2$.
\end{proof}

\begin{remark}
	If $f: (X_1,[\eta_1]) \to (X_2,[\eta_2])$ is a symplectomorphism such that $f^*[\eta_2] = [\eta_1]$,
	then the graph $\Gamma_f \subseteq X_1^- \times X_2$ is also restricted.
	By replacing $\Delta$ with $\Gamma_f$, Theorem \ref{theorem: morita-equivalence} generalizes to show that if $\OO_h^1$ and $\OO_h^2$ are Frobenius-constant quantizations of $X_1$ and $X_2$ and $f^*[(\OO^2_h)^\sharp] = [(\OO_h^1)^\sharp]$,
	then $f^*\OO_h^2$ is Morita equivalent to $\OO_h^1$.
\end{remark}


\printbibliography

\noindent
{\sc University of Chicago\\
Chicago, IL, USA}

\bigskip

\noindent
{\em E-mail address\/}: {\tt mundinger@uchicago.edu}\\

\end{document}